\documentclass[11pt, a4paper]{article}
\usepackage{times}
\usepackage{a4wide}
\usepackage[dvips, hyperindex]{hyperref}
\usepackage[british]{babel}
\usepackage{enumerate, longtable}
\usepackage{amsmath, amscd, amsfonts, amsthm, amssymb, latexsym, comment, stmaryrd, graphicx}
\usepackage[all]{xy}
\usepackage[T1]{fontenc}
\usepackage[latin1]{inputenc}
\usepackage{color}

\newtheorem{thm}{Theorem}[section]
\newtheorem{lem}[thm]{Lemma}
\newtheorem{defi}[thm]{Definition}

\newtheorem{prop}[thm]{Proposition}
\newtheorem{cor}[thm]{Corollary}

\newtheorem{setup}[thm]{Set-up}

\newcommand{\GL}{\mathrm{GL}}
\newcommand{\GSp}{\mathrm{GSp}}
\newcommand{\PGSp}{\mathrm{PGSp}}
\newcommand{\PSp}{\mathrm{PSp}}
\newcommand{\Sp}{\mathrm{Sp}}

\newcommand{\proj}{\mathrm{proj}}

\newcommand{\wild}{\mathrm{w}}

\newcommand{\id}{\mathrm{id}}

\DeclareMathOperator{\Sym}{Sym}

\DeclareMathOperator{\Image}{im}

\DeclareMathOperator{\Gal}{Gal}

\newcommand{\Ind}{{\rm Ind}}

\newcommand{\Res}{{\rm Res}}

\newcommand{\lhdeq}{\trianglelefteq}

\newcommand{\calL}{\mathcal{L}}

\newcommand{\FF}{\mathbb{F}}

\newcommand{\NN}{\mathbb{N}}

\newcommand{\QQ}{\mathbb{Q}}

\newcommand{\ZZ}{\mathbb{Z}}

\newcommand{\Qbar}{\overline{\QQ}}
\newcommand{\Zbar}{\overline{\ZZ}}
\newcommand{\Fbar}{\overline{\FF}}

\begin{document}

\selectlanguage{british}

\title{Compatible systems of symplectic Galois representations and the inverse
Galois problem II.\\
Transvections and huge image.}
\author{
Sara Arias-de-Reyna\footnote{Universit\'e du Luxembourg,
Facult\'e des Sciences, de la Technologie et de la Communication,
6, rue Richard Coudenhove-Kalergi,
L-1359 Luxembourg, 
Luxembourg, sara.ariasdereyna@uni.lu},
Luis Dieulefait\footnote{Departament d'\`Algebra i Geometria,
Facultat de Matem\`atiques,
Universitat de Barcelona,
Gran Via de les Corts Catalanes, 585,
08007 Barcelona, Spain, ldieulefait@ub.edu},
Gabor Wiese\footnote{Universit\'e du Luxembourg,
Facult\'e des Sciences, de la Technologie et de la Communication,
6, rue Richard Coudenhove-Kalergi,
L-1359 Luxembourg, Luxembourg, gabor.wiese@uni.lu}}
\maketitle

\begin{abstract}
This article is the second part of a series of three articles about
compatible systems
of symplectic Galois representations and applications to the inverse
Galois problem.

This part is concerned with symplectic Galois representations having
a huge residual image, by which we mean that a symplectic group
of full dimension over the prime field is contained  up to conjugation.
A key ingredient is a classification of symplectic representations
whose image contains a nontrivial transvection:
these fall into three very simply describable classes, the reducible ones,
the induced ones and those with huge image.
Using the idea of an $(n,p)$-group of Khare, Larsen and Savin
we give simple conditions under which a symplectic
Galois representation with coefficients in a finite field has a huge image.
Finally, we combine this classification result with the main result of
the first part
to obtain a strenghtened application to the inverse Galois problem.

\end{abstract}

\section{Introduction}

This article is the second of a series of three about compatible systems
of symplectic Galois representations and applications to the inverse
Galois problem.

This part is concerned with symplectic Galois representations having
a {\em huge image}: For a prime~$\ell$,
a finite subgroup $G \subseteq \GSp_n(\Fbar_\ell)$ is called {\em huge} if it
contains a conjugate (in $\GSp_n(\Fbar_\ell)$) of $\Sp_n(\FF_\ell)$.
By Corollary~\ref{cor:huge-equivalence} below this notion is the same as the one introduced
in Part~I~\cite{partI}.

Whereas the classification of the finite subgroups of $\Sp_n(\Fbar_\ell)$
appears very complicated to us, it turns out that the finite subgroups
containing a nontrivial transvection can be very cleanly classified
into three classes, one of which is that of huge subgroups
(see Theorem~\ref{thm:gp} below).
Translating this group theoretic result into the language of
symplectic representations whose image contains a nontrivial transvection,
these also fall into three very simply describable classes: the reducible ones,
the induced ones and those with huge image (see Corollary~\ref{cor:rep}).

Using the idea of an $(n,p)$-group of~\cite{KLS1} (i.e.\ of a maximally induced
place of order~$p$, in the terminology of Part~I), some number theory
allows us to give very simple conditions under which a symplectic
Galois representation with coefficients in $\Fbar_\ell$ has huge
image (see Theorem~\ref{thm:principal} below).

This second part is independent of the first, except for Corollary~\ref{cor:principal},
which combines the main results of Part~I~\cite{partI}, and the present Part~II.
In Part~III~\cite{partIII} written in collaboration with Sug Woo Shin,
a compatible system satisfying the assumptions of Corollary~\ref{cor:principal}
is constructed.

\subsection*{Statement of the results}

In order to fix terminology, we recall some standard definitions.
Let $K$ be a field. An $n$-dimensional $K$-vector space~$V$ equipped with a symplectic form
(i.e.\ nonsingular and alternating), denoted by $\langle v,w \rangle = v \bullet w$
for $v,w \in V$, is called a {\em symplectic $K$-space}.
A $K$-subspace $W \subseteq V$ is called a {\em symplectic $K$-subspace} if the restriction of
$\langle \cdot,\cdot \rangle$ to $W \times W$ is nonsingular (hence, symplectic).
The {\em general symplectic group} $\GSp(V,\langle \cdot,\cdot\rangle) =: \GSp(V)$
consists of those $A \in \GL(V)$ such that there is $\alpha \in K^\times$,
the {\em multiplier} (or {\em similitude factor}) of~$A$,
such that we have $(Av) \bullet (Aw) = \alpha(v \bullet w)$ for all $v,w \in V$.
The {\em symplectic group} $\Sp(V,\langle \cdot,\cdot\rangle) =: \Sp(V)$
is the subgroup of~$\GSp(V)$ of elements with multiplier~$1$.
An element $\tau \in \GL(V)$ is a {\em transvection} if $\tau - \id_V$ has rank~$1$,
i.e.\ if $\tau$ fixes a hyperplane pointwisely, and there is a line $U$ such that
$\tau(v)-v \in U$ for all~$v \in V$.
We will consider the identity as a ``trivial transvection''.
Any transvection has determinant and multiplier~$1$.
A {\em symplectic transvection} is a transvection in~$\Sp(V)$.
Any symplectic transvection has the form
$$ T_v[\lambda] \in \Sp(V): u \mapsto u + \lambda (u\bullet v)  v$$
with {\em direction vector} $v \in V$ and {\em parameter} $\lambda \in K$
(see e.g.\ \cite{A}, pp.~137--138).

The classification result on subgroups of general symplectic groups containing
a nontrivial trans\-vection which plays the key role in our approach is the following.

\begin{thm}\label{thm:gp}
Let $K$ be a finite field of characteristic at least~$5$ and $V$ a symplectic
$K$-vector space of dimension~$n$.
Then any subgroup~$G$ of $\GSp(V)$ which contains a nontrivial symplectic transvection
satisfies one of the following assertions:
\begin{enumerate}[1.]
\item\label{thm:gp:1} There is a proper $K$-subspace $S \subset V$ such that $G(S) = S$.
\item\label{thm:gp:2} There are mutually orthogonal nonsingular symplectic $K$-subspaces~$S_i \subset V$ with $i=1,\dots,h$
of dimension~$m$ for some $m < n$, such that
$V = \bigoplus_{i=1}^h S_i$ and for all $g \in G$ there is a permutation $\sigma_g \in \Sym_h$
(the symmetric group on $\{1,\dots,h \}$)
with $g(S_i) = S_{\sigma_g(i)}$. Moreover, the action of $G$ on the set $\{S_1, \dots, S_h\}$ thus defined is transitive.
\item\label{thm:gp:3} There is a subfield L of K such that the subgroup generated by the symplectic transvections of $G$ is conjugated (in $\GSp(V)$) to $\Sp_{n}(L)$.
\end{enumerate}
\end{thm}

In Section~\ref{sec:section1} we show how this theorem can be deduced from results of Kantor~\cite{Kantor1979}. In a previous version of this article we had given a self-contained proof,
which is still available on arXiv.
For our application to Galois representations we provide the following representation
theoretic reformulation of Theorem~\ref{thm:gp}.

\begin{cor}\label{cor:rep}
Let $\ell$ be a prime at least~$5$, let $\Gamma$ be a compact topological group and
$$ \rho: \Gamma \to \GSp_{n}(\Fbar_\ell)$$
a continuous representation (for the discrete topology on $\Fbar_\ell$).
Assume that the image of $\rho$ contains a nontrivial transvection.
Then one of the following assertions holds:
\begin{enumerate}[1.]
\item\label{cor:rep:1} $\rho$ is reducible.
\item\label{cor:rep:2} There is a closed subgroup $\Gamma' \subsetneq \Gamma$ of finite index~$h \mid n$ and a
representation $\rho': \Gamma' \to \GSp_{n/h}(\Fbar_\ell)$ such that
$\rho \cong \Ind_{\Gamma'}^\Gamma (\rho')$.
\item\label{cor:rep:3} There is a finite field $L$  of characteristic $\ell$  such that the subgroup generated by
the symplectic transvections in the image of~$\rho$ is conjugated (in $\GSp_n(\overline{\mathbb{F}}_{\ell})$) to
$\Sp_{n}(L)$; in particular, the image is huge.
\end{enumerate}
\end{cor}

The following corollary shows that the definition of a huge subgroup
of $\GSp_{n}(\Fbar_\ell)$, which we give in Part~I~\cite{partI},
coincides with the simpler definition stated above.

\begin{cor}\label{cor:huge-equivalence}
Let $K$ be a finite field of characteristic $\ell\geq 5$,
$V$ a symplectic $K$-vector space of dimension $n$,
and $G$ a subgroup of $\GSp(V)$ which contains a
symplectic transvection. Then the following are equivalent:

\begin{enumerate}[(i)]
\item\label{cor:huge-equivalence:i} $G$ is huge.

\item\label{cor:huge-equivalence:ii} There is a subfield $L$ of $K$ such that the subgroup generated by
the symplectic transvections of $G$ is conjugated (in $\GSp(V)$) to
$\Sp_{n}(L)$.
\end{enumerate}
\end{cor}

Combining the group theoretic results above with $(n,p)$-groups, introduced by~\cite{KLS1},
some number theory allows us to prove the following theorem. Before stating it, let us collect some notation.

\begin{setup}\label{setup}
Let $n,N\in \NN$ be integers with $n$ even and $N=N_1\cdot N_2$ with $\gcd(N_1, N_2)=1$.
Let $L_0$ be the compositum of all number fields of degree $\leq n/2$,
which are ramified at most at the primes dividing $N_2$ (which is a number field).
Let $q$ be a prime which is completely split in $L_0$,
and let $p$ be a prime dividing $q^n-1$ but not dividing $q^{\frac{n}{2}}-1$,
and $p\equiv 1 \pmod n$.
\end{setup}

\begin{thm}\label{thm:principal} Assume Set-up~\ref{setup}. Let $k\in \mathbb{N}$, $\ell\not=p, q$ be a prime such that $\ell>kn!+1$ and $\ell\nmid N$.
Let $\chi_q:G_{\QQ_{q^n}}\rightarrow \Qbar_{\ell}^{\times}$ be a character
satisfying the assumptions of Lemma~\ref{lem:np}, and $\overline{\chi}_q$ 
the composition of $\chi_q$  with the reduction map $\Zbar_{\ell}\rightarrow \Fbar_{\ell}$. Let $\overline{\alpha}:G_{\QQ_{q}}\rightarrow \overline{\mathbb{F}}_{\ell}^{\times}$ be an unramified character.

Let
$$\rho:G_{\QQ}\rightarrow \GSp_n(\Fbar_{\ell})$$
be a Galois representation, ramified only at the primes dividing
$Nq\ell$, satisfying that a twist by some power of the cyclotomic
character is regular in the sense of Definition \ref{defi:regularity} with tame inertia weights at most $k$, and
such that (1) $\Res^{G_{\QQ}}_{G_{\QQ_q}}(\rho)=\Ind^{G_{\QQ_q}}_{G_{\QQ_{q^n}}}(\overline{\chi}_q)\otimes \overline{\alpha}$, (2)
the image of $\rho$ contains a nontrivial transvection and (3) for all primes $\ell_1$ dividing $N_1$,
the image under $\rho$ of $I_{\ell_1}$, the inertia group at $\ell_1$, has order prime to $n!$.

Then the image of $\rho$ is a huge subgroup of $\GSp_n(\Fbar_{\ell})$.
\end{thm}

Combining Theorem~\ref{thm:principal} with the results of Part~I~\cite{partI} of this series
yields the following corollary.

\begin{cor}\label{cor:principal} Assume Set-up~\ref{setup}.
Let $\rho_\bullet = (\rho_\lambda)_\lambda$ (where $\lambda$ runs through the finite places of a number field $L$) be an $n$-dimensional a.~e. absolutely irreducible a.~e.
symplectic compatible system, as defined in Part~I (\cite{partI}), for the
base field~$\QQ$, which satisfies the following assumptions:
\begin{itemize}
\item For all places~$\lambda$ the representation $\rho_\lambda$ is unramified outside $Nq\ell$,
where $\ell$ is the rational prime below~$\lambda$.
\item There are $a \in \ZZ$ and $k \in \NN$ such that, for all but possibly finitely many places $\lambda$ of $L$, the reduction mod $\lambda$ of $\chi_{\ell}^a\otimes \rho_{\lambda}$ is regular in the sense of Definition \ref{defi:regularity}, with tame inertia weights at most~$k$.
\item The multiplier of the system is a finite order character times a power of the
cyclotomic character.
\item For all primes $\ell$ not belonging to a density zero set of rational primes, and for each $\lambda\vert \ell$,
the residual representation $\overline{\rho}_\lambda$ contains a nontrivial transvection
in its image.
\item For all places $\lambda$ not above~$q$ one has $\Res^{G_{\QQ}}_{G_{\QQ_q}}(\rho_\lambda)=\Ind^{G_{\QQ_q}}_{G_{\QQ_{q^n}}}(\chi_q)\otimes \alpha$,
where $\alpha: G_{\QQ_q}\rightarrow \overline{L}_\lambda^{\times}$ is some unramified character
and $\chi_q:G_{\QQ_{q^n}}\rightarrow \Zbar^{\times}$ is a character
such that its composite with the embedding $\Zbar^\times \hookrightarrow \Qbar_\ell^\times$ given by $\lambda$ satisfies the assumptions of Lemma~\ref{lem:np} for all primes~$\ell \nmid pq$.
In the terminology of Part~I, $q$ is called a maximally induced place of order~$p$.

\item For all primes $\ell_1$ dividing $N_1$ and for all but possibly finitely many places~$\lambda$,
the group $\overline{\rho}_\lambda (I_{\ell_1})$ has order prime to $n!$ (where $I_{\ell_1}$ denotes the inertia group at $\ell_1$).
\end{itemize}

\noindent Then we obtain:
\begin{enumerate}[(a)]
\item\label{cor:principal:a} For all primes $\ell$ not belonging to a density zero set of rational primes, and for each $\lambda\vert \ell$, the image of
the residual representation $\overline{\rho}_\lambda$ is a huge subgroup of $\GSp_n(\Fbar_{\ell})$.
\item\label{cor:principal:b} For any $d \mid \frac{p-1}{n}$ there exists a set~$\calL_d$ of rational primes~$\ell$ of positive density such that for all $\ell \in \calL_d$ there is a place $\lambda$ of~$L$
above~$\ell$ satisfying that the image of $\overline{\rho}_\lambda^\proj$
is $\PGSp_n(\FF_{\ell^d})$ or $\PSp_n(\FF_{\ell^d})$.
\end{enumerate}
\end{cor}

The proofs of Theorem~\ref{thm:principal} and Corollary \ref{cor:principal}
are given in Section~\ref{sec:section2}.

\subsection*{Acknowledgements}

S.~A.-d.-R.\ worked on this article as a fellow of the Alexander-von-Humboldt foundation.
She thanks the Universit\'e du Luxembourg for its hospitality during a long term visit in 2011.
She was also partially supported by the project MTM2012-33830 of the Ministerio de Econom\'ia y Competitividad of Spain.
L.~V.~D. was supported by the project MTM2012-33830 of the Ministerio de Econom\'ia y Competitividad of Spain and by an ICREA Academia Research Prize.
G.~W.\ was partially supported by the DFG collaborative research centre TRR~45,
the DFG priority program~1489 and the Fonds National de la Recherche Luxembourg (INTER/DFG/12/10).
S.~A.-d.-R. and G.~W. thank the Centre de Recerca Matem\`{a}tica  for its support and hospitality during a long term visit in 2010.

The authors thank the anonymous referee for his suggestions on Section~\ref{sec:section1}.
S.~A.-d.-R.\ and G.~W.\ thank Gunter Malle for his detailed explanations concerning Kantor's paper~\cite{Kantor1979}.

\section{Symplectic representations containing a transvection}\label{sec:section1}

In this section we deduce Theorem \ref{thm:gp} from the results of Kantor \cite{Kantor1979}, together with some representation theory of groups. Throughout the section our setting will be the following: $\ell\geq 5$ denotes a prime number,  $n$ an even positive integer and $V$ a symplectic $n$-dimensional vector space over a finite field $K$ of characteristic $\ell$.

\subsection{Kantor's classification result}

In his paper \cite{Kantor1979}, Kantor classifies subgroups of classical linear groups which are generated by a conjugacy class of elements of long root subgroups. In this paper we are only concerned with subgroups of the symplectic group $\Sp(V)$. This case is addressed in $\S 11$ of~\cite{Kantor1979}.

We need some notation in order to state his result.
First of all, recall that in the symplectic case, the elements of long root subgroups are precisely the symplectic transvections. Given a subgroup $H\subseteq \Sp(V)$, denote by $O_{\ell}(H)$ the maximal normal $\ell$-subgroup contained in $H$, denote by $[H, H]$ the commutator subgroup of $H$, and by $Z_{\Sp(V)}(H)$ the centraliser of $H$ in $\Sp(V)$. Below we state the result of Kantor in the symplectic case (and leaving aside the cases of characteristic $2$ and $3$).

\begin{thm}[Kantor]\label{thm:Kantor}
Assume that $\ell\geq 5$, and let $H\subseteq \Sp(V)$ be a subgroup satisfying the following conditions:
\begin{enumerate}
 \item\label{cond:1} There exists a set $\mathfrak{X}\subseteq H$ consisting of transvections, closed under conjugation in $H$, which generates $H$.
 \item\label{cond:2} $O_{\ell}(H)\leq [H, H]\cap Z_{\Sp(V)}(H)$.
 \item\label{cond:3} $H$ does not preserve any nonsingular subspace of~$V$.
\end{enumerate}
Then there is a subfield $L$ of $K$ such that $H$ is conjugated (in $\Sp(V)$) to $\Sp_n(L)$.
\end{thm}

We will apply this result in the case when $H$ is an irreducible subgroup. In this case, Conditions \ref{cond:2} and \ref{cond:3} are satisfied. We elaborate on Condition \ref{cond:2}. Let $W\subseteq V$ be the subspace of elements that are left invariant by all elements in $O_{\ell}(H)$. Since $O_{\ell}(H)$ is an $\ell$-group acting on a finite $\ell$-group $V$, the cardinality of $W$ is divisible by $\ell$ (cf.~Lemma 1 of Chapter IX of \cite{LocalFields}), hence $W\not=\{0\}$. Moreover, since $O_{\ell}(H)$ is a normal subgroup of $H$, it follows that $H$ stabilises $W$. But $H$ is an irreducible group, hence $W=V$ and $O_{\ell}(H)=\{\mathrm{Id}\}$.
Furthermore, if we take into account that the conjugate of a transvection is again a transvection, we can reformulate Condition \ref{cond:1} as follows: ``the transvections contained in $H$ generate $H$'', or simply ``$H$ is generated by transvections''. This discussion proves the following corollary.

\begin{cor}\label{cor:Kantor}
Assume that $\ell\geq 5$, and let $H\subseteq \Sp(V)$ be an irreducible subgroup which is generated by transvections. Then there is a subfield $L$ of $K$ such that $H$ is conjugated (in $\Sp(V)$) to $\Sp_n(L)$.
\end{cor}

\subsection{Proof of the group theoretic results}

We will make use of the following facts about transvections, the simple proofs
of which are omitted.

\begin{lem}\label{lem:aux}
Let $T_u[\lambda]\in \Sp(V)$  be a symplectic transvection. Then 
\begin{enumerate}[(a)]
\item\label{lem:aux:a} For any $A\in \GSp(V)$ with multiplier $\alpha\in K^{\times}$, $AT_u[\lambda]A^{-1}=T_{Au}[\frac{\lambda}{\alpha}]$.
\item\label{lem:aux:b}  Suppose $W\subseteq V$ is a $K$-vector subspace stabilised by $T_u[\lambda]$
with $\lambda \in K^\times$. Then we have
\begin{enumerate}[(1)]
\item\label{lem:aux:b:1} $u\in W$ or $u \in W^{\perp}$;
\item\label{lem:aux:b:2} $u \in W^{\perp} \Leftrightarrow T_u[\lambda]|_W = \id_W$.
\end{enumerate}
\end{enumerate}
\end{lem}

\begin{proof}[Proof of Theorem~\ref{thm:gp}]
Let $G\subseteq \GSp(V)$ be a subgroup which contains a nontrivial transvection.
If the action of $G$ on $V$ is reducible, we are in case~\ref{thm:gp:1} of the theorem.
Assume that the action of $G$ on $V$ is irreducible, and define the subgroup
$H:=\langle \tau\in G: \tau\text{ is a transvection}\rangle$. Note that $H$ is nontrivial.
If the action of $H$ on $V$ is irreducible, then we can apply Corollary~\ref{cor:Kantor} 
to the group $H$ and conclude that $H$ is conjugate in $\GSp(V)$ to $\Sp_n(L)$ for 
some subfield $L\subseteq K$. This is case~\ref{thm:gp:3} of the theorem.

Assume then that the action of $H$ on $V$ is reducible. Let $W\subset V$ be a $K$-vector
subspace on which $H$ acts irreducibly. By Lemma~\ref{lem:aux}(\ref{lem:aux:a}),
the group $H$ is a normal subgroup of~$G$.
Thus we can apply Clifford's Theorem (cf.~\cite{CR}, (11.1)), to obtain
$g_1, \dots, g_r\in G$ such that we have the equality of $H$-modules
\begin{equation}\label{eq1}
 V=\bigoplus_{i=1}^r g_iW.
\end{equation}

We first remark that $W$ is not the trivial $H$-module, as otherwise $H$ would act trivially
on~$V$ and thus $H$ would be the trivial group.
Now consider $W' = \langle u \in W : \exists \lambda \in K^\times: T_u[\lambda] \in H\rangle$.
As $W$ is a nontrivial $H$-module, $W' \neq 0$. Let $T_v[\mu] \in H$ and $u \in W'$.
By Lemma~\ref{lem:aux}(\ref{lem:aux:b}), $v \in W'$ or $v \in W^\perp$. In both
cases we have $T_v[\mu](u)=u+\mu (u\bullet v)v \in W'$, showing that $H$ preserves~$W'$,
so that the irreducibility of~$W$ implies $W'=W$.

Let $\tilde{W}=gW$ be a conjugate of~$W$ for which we assume $\tilde{W}\neq W$,
so that $\tilde{W}\cap W=0$ by the irreducibility.
We have just seen that there are $w_1,\dots,w_m \in W$ spanning~$W$ and
$\lambda_1,\dots,\lambda_m \in K^\times$ such that
$T_{w_1}[\lambda_1],\dots,T_{w_m}[\lambda_m] \in H$.
As $H$ also preserves $\tilde{W}$, Lemma~\ref{lem:aux}(\ref{lem:aux:b})
shows $w_i \in \tilde{W}^\perp$ for $1 \le i \le m$.
This proves two things. Firstly, $W \subseteq \tilde{W}^\perp$ and this means that the decomposition~\eqref{eq1} of~$V$ is into mutually orthogonal spaces. From this it follows
that these subspaces are also symplectic, i.e.\ that the pairing is nondegenerate on each subspace.
Secondly, $T_{w_1}[\lambda_1]$ is the identity on~$\tilde{W}$, but it is nontrivial on~$W$
(e.g.\ by the nondegeneration of $W$ there is $u \in W$ such that $u \bullet w_1 \neq 0$,
whence $T_{w_1}[\lambda_1](u) \neq u$). Hence, $W$ and $\tilde{W}$ are nonisomorphic as $H$-modules.

Considering the composite maps
$gW \hookrightarrow V \xrightarrow{\textnormal{projection}}g_iW$,
in view of the irreducibility of the $g_iW$ and the fact $g_iW \not\cong g_jW$ for
$i\neq j$, it follows that $gW$ is one of the~$g_iW$.
Thus, $G$ acts on the set $\{g_1W,\dots,g_rW\}$. If this action
were not transitive, then the sum of the spaces in one orbit would be a proper nontrivial
$G$-submodule of~$V$, contradicting the irreducibility of~$V$.
Thus, all statements of case~\ref{thm:gp:2} of the theorem are proved.
\end{proof}

\begin{proof}[Proof of Corollary~\ref{cor:rep}.]
Since $\Gamma$ is compact and the topology on~$\Fbar_\ell$ is discrete,
the image of $\rho$ is a subgroup of $\GSp_{n}(K)$ for a certain finite field~$K$
of characteristic~$\ell$.
Therefore one of the three possibilities of Theorem~\ref{thm:gp} holds for
$G:=\Image(\rho)$. If the first holds, then $\rho$ is reducible, and if the third holds,
then $\Image(\rho)$ contains a group conjugate to $\Sp_{n}(L)$ for some subfield $L$ of~$K$.

Assume now that the second possibility holds. We use notation as in Theorem~\ref{thm:gp}.
Let $\Gamma'$ be $\{g \in \Gamma \;|\; \sigma_g(1)=1\}$, the stabiliser of the first subspace.
This is a closed subgroup of~$\Gamma$ of finite index.
Choose coset representatives and write $\Gamma = \bigsqcup_{i=1}^{h'} g_i \Gamma'$.
The set $\{\gamma S_1 \;|\; \gamma \in \Gamma\}$ contains $h'$ elements, namely precisely the
$g_i S_1$ for $i=1,\dots,h'$. As the action of $G$ on the decomposition is transitive,
this set is precisely $\{S_1,\dots,S_h\}$, whence $h=h'$.
Define $\rho'$ as the restriction of~$\rho$ to~$\Gamma'$ acting on~$S_1$.
Then as $\Gamma'$-representation we have the isomorphism
$$ V  \cong \bigoplus_{i=1}^h S_i \cong \bigoplus_{i=1}^h g_i S_1.$$
Proposition (10.5) of $\S 10A$ of \cite{CR} implies that $\rho=\Ind^{\Gamma}_{\Gamma'}(\rho')$.
\end{proof}

\begin{proof}[Proof of Corollary~\ref{cor:huge-equivalence}.]
Assume that $G$ contains a subgroup conjugate (in
$\GSp(V)$) to $\Sp_{n}(\FF_{\ell})$. In particular, $G$ does not fix
any proper subspace $S\subset V$, nor any decomposition
$V=\bigoplus_{i=1}^h S_i$ into mutually orthogonal nonsingular
symplectic subspaces. Hence by Theorem~\ref{thm:gp} there is a
subfield L of K such that the subgroup generated by the symplectic
transvections of $G$ is conjugated (in $\GSp(V)$) to $\Sp_{n}(L)$.
The other implication is clear.
\end{proof}

\section{Symplectic representations with huge image}\label{sec:section2}

In this section we establish Theorem~\ref{thm:principal}.

\subsection{$(n,p)$-groups}

As a generalisation of dihedral groups, in \cite{KLS1}, Khare, Larsen and
Savin introduce so-called $(n, p)$-groups.
We briefly recall some facts and some notation to be used.
For the definition of $(n, p)$-groups we refer to~\cite{KLS1}.
Let $q$ be a prime number, and let $\QQ_{q^n}/\QQ_q$ be the unique unramified extension
of $\QQ_q$ of degree $n$ (inside a fixed algebraic closure $\Qbar_q$).
Assume $p$ is a prime such that the order of $q$ modulo
$p$ is~$n$. Recall that $\QQ_{q^n}^{\times}\simeq\mu_{q^n-1}\times U_1\times q^{\ZZ}$,
where $\mu_{q^n-1}$ is the group of $(q^n-1)$-th roots of unity and $U_1$ the group
of $1$-units.
Let $\ell$ be a prime distinct from $p$ and~$q$.
Assuming that $p, q>n$, in \cite{KLS1} the authors construct a character $\chi_q:\QQ_{q^n}^{\times}\rightarrow \overline{\QQ}_{\ell}^{\times}$
that satisfies the three properties of the following lemma, which
is proved in \cite{KLS1}, Section~3.1.

\begin{lem}\label{lem:np}
Let $\chi_q:\QQ_{q^n}^{\times}\rightarrow \overline{\QQ}_{\ell}^{\times}$ be a character
satisfying:
\begin{itemize}
\item $\chi_q$ has order $2p$.
\item $\chi_q\vert_{\mu_{q^n-1}\times U_1}$ has order $p$.
\item $\chi_q(q)=-1$.
\end{itemize}
This character gives rise to a character (which by abuse of notation we call also $\chi_q$) of $G_{\QQ_{q^n}}$ by means of the reciprocity map of local class field theory.

Let $\rho_q=\Ind^{G_{\QQ_q}}_{G_{\QQ_{q^n}}}(\chi_q)$.
Then $\rho_q$ is irreducible and symplectic, in the sense that it can be conjugated to take values in $\Sp_n(\Qbar_{\ell})$,
and the image of the reduction $\overline{\rho}_q$ of $\rho_q$ in $\Sp_{n}(\Fbar_{\ell})$ is an $(n, p)$-group. Moreover, if $\overline{\alpha}:G_{\QQ_q}\rightarrow \overline{\mathbb{F}}_{\ell}^{\times}$ is an unramified character, then $\overline{\rho}_q\otimes \overline{\alpha}$ is also irreducible.
\end{lem}

Note that also the reduction of $\rho_q$ is $\Ind^{G_{\QQ_q}}_{G_{\QQ_{q^n}}}(\overline{\chi}_q)$,
which is an irreducible representation.
Here $\overline{\chi}_q$ is the composite of $\chi_q$ and the projection
$\Zbar_\ell \twoheadrightarrow \Fbar_\ell$.
To see why the last assertion is true, note that to see that $\overline{\rho}_q\otimes \overline{\alpha}=\Ind^{G_{\QQ_q}}_{G_{\QQ_{q^n}}}(\overline{\chi}_q\otimes (\overline{\alpha}\vert_{G_{\QQ_{q^n}}}))$ is irreducible, it suffices to prove that the $n$ characters $\overline{\chi}_q\otimes(\overline{\alpha}\vert_{G_{\QQ_{q^n}}}), (\overline{\chi}_q\otimes (\overline{\alpha}\vert_{G_{\QQ_{q^n}}}))^q, \dots, (\overline{\chi}_q\otimes (\overline{\alpha}\vert_{G_{\QQ_{q^n}}}))^{q^{n-1}}$ are different (cf.\ \cite{Serre:LinearGroups}, Proposition 23, Chapter 7). But the order of the restriction of $\overline{\chi}_q\otimes (\overline{\alpha}\vert_{G_{\QQ_{q^n}}})$ to the inertia group at $q$ is $p$ (since $\overline{\alpha}$ is unramified), and the order of $q$ mod $p$ is $n$.

\subsection{Regular Galois representations}

In our result we assume that our representation $\rho$ is regular, which is a condition on the tame inertia weights of $\rho$.

\begin{defi}[Regularity]\label{defi:regularity}
Let $\ell$ be a prime number, $n$ a natural number, $V$ an $n$-dimensional vector space over $\overline{\mathbb{F}}_{\ell}$  and $\rho:G_{\QQ_{\ell}}\rightarrow \GL(V)$ a Galois representation,
and denote by $I_{\ell}$ the inertia group at $\ell$.
We say that $\rho$ is \emph{regular} if there exists an integer~$s$ between $1$ and $n$,
and for each $i=1, \dots, s$,
a set $S_i$ of natural numbers in $\{0,1,\dots,\ell-1\}$, of cardinality $r_i$,
with $r_1 + \cdots + r_s=n$, say $S_i=\{a_{i, 1}, \dots, a_{i, r_i}\}$,
such that the cardinality of $S = S_1\cup\dots\cup S_s$ equals~$n$
(i.e.\ all the $a_{i,j}$ are distinct) and such that,
if we denote by $B_i$ the matrix
$$B_i\sim \begin{pmatrix} \psi_{r_i}^{b_i} & \ & \ & 0 \\
\ & \psi_{r_i}^{b_i\ell} & \ & \ \\ \ & \ & \ddots & \ \\ 0 & \ & \
& \psi_{r_i}^{b_i\ell^{r_i-1}}\end{pmatrix}$$
with $\psi_{r_i}$ our fixed choice of fundamental character of niveau~$r_i$ and
$b_i=a_{i, 1} + a_{i, 2}\ell + \cdots + a_{i, r_i}\ell^{r_i-1}$,
then
$$\rho\vert_{I_{\ell}}\sim\left(\begin{array}{c  c  c }
B_1 & \multicolumn{1}{|c}{} & * \\
\cline{1-1}
\  & \ddots & \ \\
\cline{3-3}
0 & \ & \multicolumn{1}{|c}{B_s} \\
\end{array}\right).$$ The elements of $S$ are called \emph{tame inertia weights} of $\rho$.
We will say that $\rho$ has \emph{tame inertia weights at most $k$} if $S\subseteq \{0, 1, \dots, k\}$.
We will say that a global representation $\rho:{G_{\QQ}}\rightarrow \GL(V)$ is \emph{regular} if $\rho\vert_{G_{\QQ_{\ell}}}$ is regular.
\end{defi}

\begin{lem}\label{lem:technical}
Let $\rho:G_{\QQ_{\ell}}\rightarrow \GL_n(\Fbar_{\ell})$ be a  Galois representation which is regular with tame inertia weights at most $k$.
Assume that $\ell> kn!+1$.
Then all the $n!$-th powers of the characters on the diagonal of $\rho\vert_{I_{\ell}}$ are distinct.
\end{lem}

\begin{proof}
We use the notation of Definition~\ref{defi:regularity}.
Assume we had that the $n!$-th powers of two characters of the diagonal coincide, say $$\psi_{r_i}^{n!(c_0 + c_1\ell + \cdots + c_{r_i-1}\ell^{r_{i}-1})}=\psi_{r_j}^{n!(d_0 + d_1\ell + \cdots + d_{r_j-1}\ell^{r_{j}-1})},$$ where $c_0, \dots, c_{r_i-1}, d_0, \dots, d_{r_j-1}$ are distinct elements of $S_1\cup \cdots \cup S_s$.

Let $\psi_{r_ir_j}$ be a fundamental character of niveau $r_ir_j$ such that $\psi_{r_ir_j}^{\frac{\ell^{r_ir_j}-1}{\ell^{r_i}-1}}=\psi_{r_i}$ and $\psi_{r_ir_j}^{\frac{\ell^{r_ir_j}-1}{\ell^{r_j}-1}}=\psi_{r_j}$. We can write the equality above as
$$\psi_{r_ir_j}^{\frac{\ell^{r_ir_j}-1}{\ell^{r_i}-1}n!(c_0 + c_1\ell + \cdots + c_{r_i-1}\ell^{r_{i}-1})}=\psi_{r_ir_j}^{\frac{\ell^{r_ir_j}-1}{\ell^{r_j}-1}n!(d_0 + d_1\ell + \cdots + d_{r_j-1}\ell^{r_{j}-1})}.$$
In other words, $\ell^{r_ir_j}-1$ divides the quantity $$C_0=\left| \frac{\ell^{r_ir_j}-1}{\ell^{r_i}-1}n!(c_0 + c_1\ell + \cdots + c_{r_i-1}\ell^{r_{i}-1}) - \frac{\ell^{r_ir_j}-1}{\ell^{r_j}-1}n!(d_0 + d_1\ell + \cdots + d_{r_j-1}\ell^{r_{j}-1})\right|.$$ Note that $C_0$ is nonzero because modulo $\ell$ it is congruent to $n!(c_0-d_0)$, and by assumption all elements in $S_1\cup \cdots \cup S_s$ are in different congruence classes modulo $\ell$. But $\vert c_{0} + c_{1}\ell + \cdots + c_{r_i-1}\ell^{r_i-1}\vert\leq k(1 + \ell + \cdots + \ell^{r_i-1})= k (\ell^{r_i}-1)/(\ell-1)$. Analogously $\vert d_{0} + d_{1}\ell + \cdots + d_{r_j-1}\ell^{r_j-1}\vert< k (\ell^{r_j}-1)/(\ell-1)$. Thus $C_0$ is bounded above by
\begin{multline*}
\max\{\left| \frac{\ell^{r_ir_j}-1}{\ell^{r_i}-1}n!(c_0 + c_1\ell + \cdots + c_{r_i-1}\ell^{r_{i}-1})\right|, \left| \frac{\ell^{r_ir_j}-1}{\ell^{r_j}-1}n!(d_0 + d_1\ell + \cdots + d_{r_j-1}\ell^{r_{j}-1})\right| \}\\ 
\leq n!k\left(\frac{\ell^{r_ir_j}-1}{\ell-1}\right)<n!k\left(\ell^{r_ir_j-1} + 2\ell^{r_ir_j-2}\right).
\end{multline*}
Since $\ell-2\geq n!k$, we have $\ell^2-1>\ell^2-4\geq n!k(\ell+2)$ and thus $C_0<n!k(\ell^{r_ir_j-1} +2\ell^{r_ir_j-2})= n!k(\ell+2) \ell^{r_ir_j-2}< \ell^{r_ir_j}-1$.
Hence $\ell^{r_ir_j}-1$ cannot divide $C_0$.
\end{proof}

We will now use these lemmas to study the ramification at $\ell$ of an induced representation under the assumption of regularity (possibly after a twist by a power of the cyclotomic character) and boundedness of tame inertia weights.

\begin{prop}\label{prop:RamAtEll}
Let $n, m, k\in \NN$, $a\in \mathbb{Z}$ and let $\ell> kn!+1$ be a prime, $K/\QQ$ a finite extension such that $[K:\QQ]\cdot m=n$, $\rho:G_K\rightarrow \GL_m(\Fbar_{\ell})$ a Galois representation and let $\beta=\Ind_{G_K}^{G_{\QQ}}\rho$. If $\chi_{\ell}^a\otimes \beta$ is regular with tame inertia weights at most $k$, then $K/\QQ$ does not ramify at~$\ell$.
\end{prop}

\begin{proof} Assume that $K/\QQ$ ramifies at $\ell$; we will derive a contradiction. First of all, let us fix some notation: let $N/\QQ$ be the Galois closure of $K/\QQ$, and let us fix a prime $\lambda$ of $N$ above $\ell$. Denote by $I_{\ell}\subset G_{\QQ}$ the inertia group at $\ell$, $I_{\ell, \wild}\subset I_{\ell}$ the wild inertia group at $\ell$ and $I_N\subset G_{N}$ the inertia group at the prime $\lambda$. Let $W$ be the $\overline{\mathbb{F}}_{\ell}$-vector space underlying $\rho$. For each $\gamma\in G_{\QQ}$, denote $^{\gamma}K=\gamma(K)$ and define $^{\gamma}\rho:G_{^{\gamma}K}\rightarrow \GL(W)$ by $^{\gamma}\rho(\sigma)= \rho(\gamma\sigma\gamma^{-1})$.

Let us now pick any $\gamma\in G_{\QQ}$, $\sigma\in I_{\ell}$ and $\tau\in I_{N}$. Since $I_{\ell}/I_{\ell, \wild}$ is cyclic, we have that the commutator $\sigma^{-1}\tau\sigma\tau^{-1}$ belongs to $I_{\ell, \wild}$. Since $I_N\subset I_{\ell}$ is normal, $\sigma^{-1} \tau\sigma\in I_N\subset G_N\subset G_{^{\gamma}K}$, so we may apply $^{\gamma}\rho$ and conclude
\begin{equation*}
^{\gamma}\rho(\sigma^{-1} \tau \sigma)^{\gamma}\rho(\tau^{-1})=\ ^{\gamma}\rho(\sigma^{-1} \tau \sigma\tau^{-1})\in\ ^{\gamma}\rho(I_{\ell,\wild}),
\end{equation*} hence $^{\gamma}\rho(\sigma^{-1}\tau\sigma)$ and $^{\gamma}\rho(\tau)$ have exactly the same eigenvalues.

Since $N/\QQ$ ramifies in $\ell$, we may pick $\sigma\in I_{\ell}\setminus G_N$, and since $N=\prod_{\gamma\in G_{\QQ}}\ ^{\gamma}K$, there exists some $\gamma\in G_{\QQ}$ such that $\sigma\not\in G_{^{\gamma}K}$. This implies that $\beta(\sigma\gamma)(W)\cap\beta(\gamma)(W)=0$. Choose now a set of left-coset representatives $\{\gamma_1G_K, \dots, \gamma_dG_K\}$ of $G_{K}$ in $G_{\QQ}$ with $\gamma_1=\gamma$ and $\gamma_2=\sigma\gamma$; Mackey's formula (\cite{CR}, 10.13) implies that
\begin{equation*}
 \Res_{G_N}^{G_{\QQ}}\Ind_{G_K}^{G_{\QQ}}\rho=\bigoplus_{i=1}^d \Res_{G_N}^{G_{^{\gamma_i}K}} \ ^{\gamma_i}\rho.
\end{equation*} Therefore $\beta(\tau)$ is a block-diagonal matrix, where one block is $^{\gamma}\rho(\tau)$ and another block is $^{\sigma\gamma}\rho(\tau)=\ ^{\gamma}\rho(\sigma^{-1}\tau\sigma)$. But, by hypothesis, the tame inertia weights of $\chi_{\ell}^a\otimes\beta$ are bounded. By Lemma \ref{lem:technical}, we have that the $n!$-powers of the characters on the diagonal of $\chi_{\ell}^a\otimes\beta\vert_{I_{\ell}}$ are all different, which implies that the characters on the diagonal of $\beta\vert_{I_N}$ are all different. Thus  $^{\gamma}\rho(\tau)$ and  $^{\gamma}\rho(\sigma^{-1}\tau\sigma)$ cannot have the same eigenvalues for all $\tau\in I_N$.
\end{proof}

\subsection{Representations induced in two ways}

We need a proposition concerning representations induced from different subgroups
of a certain group~$G$.

\begin{prop}\label{prop:nh}
Let $G$ be a  finite group, $N\lhdeq G$, $H\le G$. Assume $(G:N)=n$, and let $p>n$ be a prime. Let $K$ be a field of characteristic coprime to $\vert G\vert$ containing all $\vert G\vert$-th roots of unity. Let $S$ be a $K[H]$-module, $\chi:N\rightarrow K^{\times}$ a character, say $\chi=\chi_1\otimes \chi_2$, where $\chi_1:N\rightarrow K^{\times}$ (resp. $\chi_2:N\rightarrow K^{\times}$) has order equal to a nontrivial power of $p$ (resp. not divisible by $p$). Assume $$\rho:=\Ind^G_H(S)=\Ind^G_N(\chi),$$ and furthermore the $n$ caracters $\{\chi_1^{\sigma}: \sigma\in G/N\}$ are different. Then $N\le H$.
\end{prop}

Following 7.2 of \cite{Serre:LinearGroups}, if $G$ is a finite group and we are given two $G$-modules $V_1$ and $V_2$, we will denote by $\langle V_1, V_2\rangle_G:=\dim \mathrm{Hom}_G(V_1, V_2)$. It is known (Lemma 2 of Chapter 7 of \cite{Serre:LinearGroups}) that, if $\varphi_1$ and $\varphi_2$ are the characters of $V_1$ and $V_2$, then $\langle V_1, V_2\rangle_G=\langle \varphi_1, \varphi_2\rangle_G:=\frac{1}{\vert G\vert}\sum_{g\in G} \varphi_1(g^{-1})\varphi_2(g)$.

Before giving the proof, we will first prove a lemma.

\begin{lem}\label{lem:ResCharacters}
Let $G$ be a group, $N\lhdeq G$ and $H\le G$ such that $(G:H)\leq n$. Let $p$ be a prime such that $p>n$, let $K$ be a field of characteristic coprime to $\vert G\vert$ containing all $\vert G\vert$-th roots of unity, and let $\chi:N\rightarrow K^{\times}$ be a character whose order is a nontrivial power of $p$. Then  $\Res^{N}_{H\cap N}\chi$ is not trivial.
\end{lem}

\begin{proof}
Assume $\Res^{N}_{H\cap N}\chi$ is trivial. Then $H\cap N\leq \ker\chi$. But $\ker \chi\leq N$, and the index $(N:\ker\chi)\geq p$. Therefore $(N:H\cap N)\geq p$. But on the other hand $p>n\geq(G:H)\geq (HN:H)=(N:N\cap H)$. Contradiction.
\end{proof}

\begin{proof}[Proof of Proposition \ref{prop:nh}.] Observe that $\rho$ is irreducible. Namely, 
there is a well-known criterion characterising when an induced representation is irreducible
(cf.\ \cite{Serre:LinearGroups}, Proposition 23, Chapter 7).
In particular, since  $N$ is normal in $G$, we have that
$\Ind^G_{N}\chi$ is irreducible if and only if
$\chi$ is irreducible (which clearly holds) and,
for all $g\in G/N$, $(\Res^G_N(\chi))^h$ is not isomorphic to 
$\Res^G_N(\chi)$. This last condition holds because the $n$ characters $\{\chi_1^{\sigma}: \sigma\in G/N\}$ are different, and $\chi_2$ has order prime to $p$.

Since $\rho$ is irreducible, we have that
$$1=\langle \rho, \rho\rangle_G=\langle\Ind^G_H(S), \Ind^G_N(\chi)\rangle_G=\langle S, \Res^G_H\Ind^G_N(\chi)\rangle_H=\cdots,$$
where in the last step we used Frobenius reciprocity.
Now we apply Mackey's formula (\cite{CR}, 10.13) on the right hand side; note that, since $N$ is normal,
$H\backslash G/N\simeq G/(H\cdot N)$:
$$\cdots=\langle S, \bigoplus_{\gamma\in G/(H\cdot N)}
\Ind^H_{H\cap N}\Res^N_{H\cap N}(\chi^{\gamma})\rangle_H
=\sum_{\gamma\in G/(H\cdot N)}\langle S,
\Ind^H_{H\cap N}\Res^N_{H\cap N}(\chi^{\gamma})\rangle_H.$$
Hence there is a unique $\gamma\in G/(H\cdot N)$ such that
$$\langle S, \Ind^H_{H\cap N}\Res^N_{H\cap N}(\chi^{\gamma})\rangle_H=1.$$
If we prove that, for all $\gamma$,
$\Ind^H_{H\cap N}\Res^N_{H\cap N}(\chi^{\gamma})$ is irreducible, then we will have that
$$S\simeq \Ind^H_{H\cap N}\Res^N_{H\cap N}(\chi^{\gamma})$$
(for some $\gamma$), hence $\dim(S)=(H:H\cap N)$.
But, on the other hand, since $\rho=\Ind^G_H(S)=\Ind^G_N(\chi)$,
we have that $\dim(S)\cdot (G:H)=(G:N)$, so
$$\dim(S)=\frac{(G:HN)(HN:N)}{(G:HN)(HN:H)}=\frac{(H:N\cap H)}{(N:N\cap H)},$$
and therefore the conclusion is that $(N:N\cap H)=1$, in other words, $N\le H$.

Therefore to conclude we only need to see that
$\Ind^H_{H\cap N}\Res^N_{H\cap N}(\chi^{\gamma})$ is irreducible.
Since conjugation by $\gamma$ plays no role here, let us just assume $\gamma=1$.
We apply again the criterion characterising when an induced representation is irreducible.
In particular, since $H\cap N$ is normal in $H$, we have that
$\Ind^H_{H\cap N}\Res^N_{H\cap N}(\chi)$ is irreducible if and only if
$\Res^N_{H\cap N}(\chi)$ is irreducible (which clearly holds) and,
for all $h\in H/N\cap H$, $(\Res^N_{H\cap N}(\chi))^h$ is not isomorphic to 
$\Res^N_{H\cap N}(\chi)$.

So pick $h\in H\setminus N$. We have $(\Res^N_{H\cap N}(\chi))^h=\Res^N_{H\cap N}(\chi^{h})$.
Assume that $\Res^N_{H\cap N}(\chi^{h})=\Res^N_{H\cap N}(\chi)$. In particular, we obtain that $\Res^N_{H\cap N}(\chi_1^{h})=\Res^N_{H\cap N}(\chi_1)$. By Lemma~\ref{lem:ResCharacters} it holds that $\chi_1=\chi_1^{h}$ as characters of~$N$.
But we know that for all $\sigma\in G/N$,
$\chi_1^{\sigma}\neq\chi_1$. Now it suffices to observe that $H/(H\cap N)\hookrightarrow G/N$.
\end{proof}

\subsection{Proofs}

Finally we carry out the proof of Theorem~\ref{thm:principal}.

\begin{lem}\label{lem:not_induced} Assume Set-up~\ref{setup}. Let $k\in \mathbb{N}$, $\ell\not=p, q$ be a prime such that $\ell>kn!+1$ and $\ell\nmid N$.
Let $\chi_q:G_{\QQ_{q^n}}\rightarrow \Qbar_{\ell}^{\times}$ be a character
satisfying the assumptions of Lemma~\ref{lem:np}, and $\overline{\chi}_q$ 
the composition of $\chi_q$ with the reduction map $\Zbar_{\ell}\rightarrow \Fbar_{\ell}$.  Let $\overline{\alpha}:G_{\QQ_{q}}\rightarrow \overline{\mathbb{F}}_{\ell}^{\times}$ be an unramified character.

Let
$\rho:G_{\QQ}\rightarrow \GSp_n(\Fbar_{\ell})$
be a Galois representation, ramified only at the primes dividing
$Nq\ell$, such that a twist by some power of the cyclotomic
character is regular in the sense of Definition \ref{defi:regularity} with tame inertia weights at most $k$, and
satisfying (1) and (3) of Theorem~\ref{thm:principal}.
Then $\rho$ is not induced from a
representation of an open subgroup $H\subsetneq G_{\mathbb{Q}}$.\end{lem}

\begin{proof}
Let $H\subset G_{\QQ}$ be an open subgroup, say of index $h$, and
$\rho':H\rightarrow \GL_{n/h}(\Fbar_{\ell})$ a representation such that
$\rho\cong \Ind^{G_{\QQ}}_{H}(\rho')$.
Call $S_1\subseteq V$ the spaces underlying $\rho'$ and~$\rho$, respectively,
so that $\rho=\Ind^{G_{\QQ}}_{H}(S_1)$.
Recall that by assumption $\Res^{G_{\QQ}}_{G_{\QQ_q}}(\rho)=\Ind^{G_{\QQ_q}}_{G_{\QQ_{q^n}}}(\overline{\chi}_q)\otimes\overline{\alpha}$.
We want to compute $\Res_{G_{\QQ_q}}^{G_{\QQ}}\Ind_{H}^{G_{\QQ}}(S_1)$.
Let us apply Mackey's formula (\cite{CR}, 10.13). By Lemma \ref{lem:np} we know that $\Res_{G_{\QQ_q}}^{G_{\QQ}}\Ind_{H}^{G_{\QQ}}(S_1)=\Ind^{G_{\QQ_q}}_{G_{\QQ_{q^n}}}(\overline{\chi}_q)\otimes\overline{\alpha}$
is irreducible, so there can only be one summand in the formula, hence
$$\Res^{G_{\QQ}}_{G_{\QQ_q}}\Ind^{G_{\QQ}}_H (S_1)
=\Ind^{G_{\QQ_q}}_{G_{\QQ_q}\cap H}\Res^{H}_{G_{\QQ_q}\cap H}(S_1),$$
and therefore
\begin{equation}\label{eq:induced}
 \Ind^{G_{\QQ_q}}_{G_{\QQ_q}\cap H}\Res^{H}_{G_{\QQ_q}\cap H}(S_1)
=\Ind^{G_{\QQ_q}}_{G_{\QQ_{q^n}}}(\overline{\chi}_q)\otimes\overline{\alpha}.
\end{equation}

We now apply Proposition~\ref{prop:nh} to Equation~\eqref{eq:induced}. Note that $\Res^{G_{\QQ}}_{G_{\QQ_q}}\rho=\Ind^{G_{\QQ_q}}_{G_{\QQ_{q^n}}}(\overline{\chi}_q)\otimes\overline{\alpha}=
\Ind^{G_{\QQ_q}}_{G_{\QQ_{q^n}}}(\overline{\chi}_q\otimes(\overline{\alpha}\vert_{G_{\QQ_{q^n}}}))$. We can write $\overline{\chi}_q\otimes(\overline{\alpha}\vert_{G_{\QQ_{q^n}}})=\overline{\chi}_1\otimes \overline{\chi}_2$, where $\overline{\chi}_1$ has order a power of $p$ and $\overline{\chi}_2$ has order prime to $p$. Note that the restriction of $\overline{\chi}_q\otimes(\overline{\alpha}\vert_{G_{\QQ_{q^n}}})$ to the inertia group $I_q$ of $G_{\QQ_{q}}$ coincides with the restriction of $\overline{\chi}_q$, which has order $p$. Thus $(\overline{\chi}_1\otimes\overline{\chi}_2)\vert_{I_q}=\overline{\chi}_q\vert_{I_q}=\overline{\chi}_1\vert_{I_q}$. Since the order of $q$ mod $p$ is $n$, we know that the  $n$ characters $\overline{\chi}_1\vert_{I_q}, \overline{\chi}_1^q\vert_{I_q}, \dots \overline{\chi}_1^{q^n}\vert_{I_q}$ are distinct. We can take
$G = \rho(G_{\QQ_q})$ in the statement of Proposition~\ref{prop:nh}, whose order is a divisor of $2np\cdot \mathrm{ord}(\overline{\alpha})$ and, hence, prime to~$\ell$. 
It thus follows that $G_{\QQ_{q^n}}\le (G_{\QQ_q}\cap H)$.

Note that, on the one hand
$$n=\dim V=\dim (\Ind_H^{G_{\QQ}} S_1)=(G_{\QQ}:H)\dim (S_1).$$
On the other hand,
\begin{equation*}
n=\dim(\Ind^{G_{\QQ_q}}_{G_{\QQ_q}\cap H}\Res^{H}_{G_{\QQ_q}\cap H}(S_1))
 = (G_{\QQ_q}:G_{\QQ_q}\cap H)\dim (S_1),
\end{equation*}
hence $(G_{\QQ}:H)=(G_{\QQ_q}:G_{\QQ_q}\cap H)$.

Let $L$ be the number field such that $H=\Gal(\Qbar/L)$.
Now $\Gal(\Qbar/L)\cap G_{\QQ_q}=\Gal(\Qbar_q/L_{\mathfrak{q}})$, where $\mathfrak{q}$
is a certain prime of $L$ above $q$ and $L_{\mathfrak{q}}$ denotes the completion of $L$
at~$\mathfrak{q}$.
The inclusion $G_{\QQ_{q^n}}\le \Gal(\Qbar_{q}/L_{\mathfrak{q}})$ means that we have the
following field inclusions:
$$\QQ_{q}\subseteq L_{\mathfrak{q}}\subseteq \QQ_{q^n}\subseteq \Qbar_q$$
and $[L_{\mathfrak{q}}:\QQ_q]=(G_{\QQ_q}:G_{\QQ_q}\cap H)=(G_{\QQ}:H)=[L:\QQ],$ hence $q$ is inert in $L/\QQ$.

Let $\ell_1$ be a prime dividing $N_1$, let $\tilde{L}/\QQ$ be a Galois closure of $L/\QQ$, $\Lambda_1$ a prime of $\tilde{L}$ above $\ell_1$ and $I_1$ the inertia group of $\Lambda_1$ over $\QQ$.
Since $\gcd(\vert \rho(I_{\ell_1})\vert, n!)=1$ and $\Gal(\tilde{L}/\QQ)$ has order dividing $n!$, we get that the projection of $\rho(I_1)\subseteq \rho(I_{\ell_1})$ into $\rho(G_{\QQ})/\rho(G_{\tilde{L}})$ is trivial. Thus, $\rho(I_1)\subseteq  \rho(G_{\tilde{L}})$.
Hence $\tilde{L}/\QQ$ is unramified at $\ell_1$ and so is $L/\QQ$.

To sum up, we know that $L$ can only be ramified at the primes dividing $Nq\ell$.
But $L$ cannot ramify at $q$ since $L_{\mathfrak{q}}\subseteq \QQ_{q^{n}}$
(and $\QQ_{q^n}$ is an unramified extension of $\QQ_q$). We just saw that $L$ cannot ramify at the primes dividing $N_1$.
We also know that $L$ cannot be ramified at~$\ell$  (cf.\ Proposition \ref{prop:RamAtEll}).
Hence $L$ only ramifies at the primes dividing $N_2$.
By the choice of~$q$, it is completely split in~$L$,
and at the same time inert in $L$. This shows $L=\QQ$ and $H=G_\QQ$.
\end{proof}

Now we can easily prove the main group theoretic result.

\begin{proof}[Proof of Theorem~\ref{thm:principal}.]
Let $G=\mathrm{Im}\rho$.
Since $G$ contains a transvection, one of the following three possibilities holds
(cf.\ Corollary~\ref{cor:rep}):

\begin{enumerate}[1.]
\item $\rho$ is reducible.

\item There exists an open subgroup $H\subsetneq G_{\QQ}$, say of index $h$ with $n/h$ even, and a representation $\rho':H\rightarrow \GSp_{n/h}(\Fbar_{\ell})$ such that $\rho\cong \Ind^{G_{\QQ}}_H \rho'$.

\item The group generated by the transvections in $G$ is conjugated (in $\GSp_n(\Fbar_{\ell})$) to $\Sp_n(\FF_{\ell^r})$ for some exponent $r$.
\end{enumerate}

By Lemma~\ref{lem:np} $G$ acts irreducibly on $V$, hence the first possibility cannot occur. By Lemma \ref{lem:not_induced}, the second possibility does not occur. Hence the third possibility holds, and this finishes the proof of the theorem.
\end{proof}

\begin{proof}[Proof of Corollary~\ref{cor:principal}.]
This follows from the main theorem of Part~I (\cite{partI}) concerning the application
to the inverse Galois problem. In order to be able to apply it, there are two things to check:

Firstly, we note that $\rho_\bullet$ is maximally induced of order~$p$ at the prime~$q$.
Secondly, the existence of a transvection in the image of $\overline{\rho}_\lambda$
together with the special shape of the representation at~$q$ allow us to conclude
from Theorem~\ref{thm:principal} that the image of $\overline{\rho}_\lambda$ is huge
for all~$\lambda\vert \ell$, where $\ell$ runs through the rational primes outside a density zero set.
\end{proof}

\bibliography{Bibliog}
\bibliographystyle{alpha}

\end{document}